\documentclass[12pt]{amsproc}
\newtheorem{theorem}{\sc Theorem}[section]
\newtheorem{lemma}[theorem]{\sc Lemma}
\newtheorem{proposition}[theorem]{\sc Proposition}
\newtheorem{corollary}[theorem]{\sc Corollary}

\newtheorem{remark}[theorem]{\sc Remark}

\usepackage{amssymb, amsthm}
\renewcommand{\>}{\rangle}
\newcommand{\<}{\langle}

\begin{document}

\author{Jo\~ao Azevedo}
\address{Department of Mathematics, University of Brasilia\\
Brasilia-DF \\ 70910-900 Brazil}
\email{J.P.P.Azevedo@mat.unb.br}

\author{Pavel Shumyatsky}
\address{Department of Mathematics, University of Brasilia\\
Brasilia-DF \\ 70910-900 Brazil}
\email{pavel@unb.br}
\thanks{Supported by CNPq and FAPDF}

\title[Monothetic Subgroups]{Compact groups with probabilistically central monothetic subgroups}
 \begin{abstract} If $K$ is a closed subgroup of a compact group $G$, the probability that randomly chosen pair of elements from $K$ and $G$ commute is denoted by $Pr(K,G)$. Say that a subgroup $K\leq G$ is $\epsilon$-central in $G$ if $Pr(\<g\>,G)\geq \epsilon$ for any $g$ in $K$. Here $\<g\>$ denotes the monothetic subgroup generated by $g\in G$. Our main result is that if $K$ is $\epsilon$-central in $G$, then there is an $\epsilon$-bounded number $e$ and a normal subgroup $T\leq G$ such that the index $[G:T]$ and the order of the commutator subgroup $[K^e,T]$ both are finite and $\epsilon$-bounded. In particular, if $G$ is a compact group for which there is $\epsilon>0$ such that $Pr(\<g\>,G)\geq \epsilon$ for any $g \in G$, then there is an $\epsilon$-bounded number $e$ and a normal subgroup $T$ such that the index $[G:T]$ and the order of $[G^e,T]$ both are finite and $\epsilon$-bounded.
\end{abstract}
\maketitle

\section{Introduction}
In this paper, all compact groups are Hausdorff topological spaces. By a subgroup of a topological group we mean a closed subgroup unless explicitly stated otherwise. If $S$ is a subset of a topological group $G$, then we denote by $\<S\>$ the subgroup (topologically) generated by $S$. A subgroup of $G$ is monothetic if it is generated by a single element.  

The Borel $\sigma$-algebra $\mathcal M$ of a compact group $G$ is the one generated by all closed subsets of $G$. We say that a measure on $(G, \mathcal M)$ is a (left) Haar measure provided $\mu$ is both inner and outer regular, $\mu(K) < \infty$ and $\mu(xE) = \mu(E)$ for all compact subsets $K$ and measurable subsets $E$ of $G$ (see \cite[Chapter 4]{hewitt-ross} or \cite[Chapter II]{nachbin}). Recall that there is a unique Haar measure $\mu$ on $(G,\mathcal M)$ such that $\mu(G)=1$.

Let $G$ be a compact group and let $K$ be a subgroup of $G$. Consider the set $C = \{(x,y) \in K \times G \, | \, xy = yx\}$. This is closed in $K\times G$ since it is the preimage of 1 under the continuous map $f : K \times G \to G$ given by $f(x,y) = [x,y]$. Denoting the normalized Haar measures of $K$ and $G$ by $\nu$ and $\mu$, respectively, the probability that a random element from $K$ commutes with a random element from $G$ is defined as $Pr(K,G) = (\nu \times \mu)(C)$. This is a well-studied concept (see in particular \cite{eb, erl, gr, gustafson, hr1, le1, le2, nath, re}).

Recently Detomi and the second author proved in \cite{ds} that if $G$ is finite and $Pr(K,G) \geq \epsilon$ for some $\epsilon >0$, then there is a normal subgroup $T$ of $G$ and a subgroup $B$ of $K$ such that the indices $[G:T]$ and $[K:B]$ and the order of the subgroup $[T,B]$ are $\epsilon$-bounded. Throughout the article we use the expression ``$(a, b, \dots )$-bounded" to mean that a quantity is bounded from above by a number depending only on the parameters $a,b,\dots$. If $B$ and $T$ are subgroups of a group $G$, we denote by $[T,B]$ the subgroup generated by all commutators $[t,b]$ with $t\in T$ and $b\in B$. In the case where $K=G$, this is a well-known theorem due to  P.M. Neumann \cite{neumann}. Conversely, if $K$ is a subgroup of a finite group $G$, and if $T\leq G$ and $B \leq K$, then $Pr(K,G)$ is bounded away from zero in terms of the indices $[G:T]$ and $[K:B]$ and the order of $[T,B]$. The present work grew out of a desire to understand the impact of similar probabilistic considerations on the structure of a compact group. Our first goal is to extend the main result of \cite{ds} to compact groups.

\begin{proposition}\label{ds adaptation}
Let $\epsilon > 0$ and let $G$ be a compact group having a subgroup $K$ such that $Pr(K,G) \geq \epsilon$. Then there is a normal subgroup $T \leq G$ and a subgroup $B \leq K$ such that the indices $[G:T]$ and $[K:B]$ and the order of $[T,B]$ are $\epsilon$-bounded. 
\end{proposition}

Let $G$ be a compact group and let $K$ be a subgroup of $G$. If $Pr(\< g \>, G) \geq \epsilon$ for every $g \in K$ we say that $K$ is $\epsilon$-central in $G$, and in the case where $K = G$ we say that $G$ is $\epsilon$-central. If $e$ is a positive integer, we denote by $G^e$ the subgroup generated by all $e$th powers of elements of $G$. Recall that a group $G$ has finite exponent $e$ if $G^e = 1$ and $e$ is the least positive number with this property.  It is easy to see that if $G$ has exponent $e$, then $Pr(\< g \>, G) \geq \frac{1}{e}$ for all $g \in G$. More generally, if $G^e \leq Z(G)$, where $Z(G)$ denotes the center of $G$, then $Pr(\<g\>,G)\geq \frac{1}{e}$ for all $g \in G$. We shall prove the following theorem.

\begin{theorem}\label{pavel}
Let $\epsilon > 0$ and assume that the subgroup $K$ is $\epsilon$-central in $G$. Then there is an $\epsilon$-bounded number $e$ and a finite-index normal subgroup $T \leq G$ such that the index $[G:T]$ and the order of $[K^e, T]$ are $\epsilon$-bounded.
\end{theorem}

The theorem implies that if $G$ is an $\epsilon$-central compact group, then there is an $\epsilon$-bounded number $e$ and a normal subgroup $T$ such that the index $[G : T]$ and the order of $[G^e, T]$ are $\epsilon$-bounded. Moreover, the exponent of the commutator subgroup $[T,T]$ is $\epsilon$-bounded. Indeed, passing to the quotient over $[G^e,T]$ we can assume that all $e$th powers of elements of $G$ centralize $T$. In particular, $T/Z(T)$ has exponent $e$, and a theorem of Zelmanov \cite{ze3} ensures that $T/Z(T)$ is locally finite. A result of Mann \cite{mann} can then be used to deduce that $[T,T]$ has finite $e$-bounded exponent.

As usual, we denote the conjugacy class of an element $x\in G$ by $x^G$. It is easy to see that if $K$ is a  subgroup of $G$ such that $|x^G|\leq n$ for every $x \in K$, then $K$ is $\frac{1}{n}$-central in $G$. More generally, let $l, n$ be positive integers and suppose that $K$ is a subgroup of a compact group $G$ such that any conjugacy class containing an $l$th power $x^l$ of an element $x \in K$ is of size at most $n$. It is not difficult to see that $K$ is $\frac{1}{ln}$-central in $G$. It turns out that this admits a converse: if $K$ is $\epsilon$-central in $G$, then there exist $\epsilon$-bounded integers $l$ and $n$ such that every conjugacy class containing an $l$th power of an element of $K$ has cardinality at most $n$. Indeed, let $x \in K$. Since $Pr(\<x\>,G) \geq \epsilon$, in view of Proposition \ref{ds adaptation} there is a normal subgroup $T$ of $G$ and a subgroup $B$ of $\<x\>$ such that the indices $[G:T]$ and $[\<x\>:B]$ and the order of $[T,B]$ are $\epsilon$-bounded. Hence, as required, there are $\epsilon$-bounded numbers $l$ and $n$ such that $[G:C_G(x^l)]\leq n$ for all $x \in K$. Therefore we have proved that
\medskip

\textit{For every $0 < \epsilon \leq 1$ there are positive integers $l$ and $n$ depending only on $\epsilon$ with the property that if $K$ is an $\epsilon$-central subgroup of the compact group $G$, then $[G:C_G(g^l)] \leq n$ for all $g \in K$.}
\medskip

Taking this into consideration, Theorem \ref{pavel} will follow from the next proposition. 

\begin{proposition}\label{proposition}
Let $G$ be a compact group and let $l, n$ be positive integers. Suppose that there is a subgroup $K$ of $G$ such that $[G:C_G(g^l)] \leq n$ for every $g \in K$. Then there exist a positive integer $e$, depending only on $l$ and $n$, and a normal subgroup $T$ of $G$, such that the index $[G:T]$ and the order of $[K^e, T]$ are $(l,n)$-bounded.
\end{proposition}

Recall that a group is said to be a BFC-group if its conjugacy classes are finite and have bounded size. A famous theorem of B. H. Neumann says that in a BFC-group the commutator subgroup $G'$ is finite \cite{bhneumann}. It follows that if $|x^G| \leq n$ for every $x \in G$, then the order of $G'$ is bounded by a number depending only on $n$. A first explicit bound for the order of $G'$ was found by J. Wiegold \cite{wiegold}, and the best known was obtained in \cite{gm} (see also \cite{essay} and \cite{segalshalev}). Proposition \ref{proposition}  is an extension of the Neumann theorem (in the particular case where $K = G$ and $l = 1$ the proof shows that we can take $T = G$ and $e = 1$). Recently, some other generalisations of Neumann's theorem have been obtained (see in particular \cite{as, dms, dierings-s}).

We end this introduction with the remark that Theorem 1.1 admits a converse:
\medskip

\textit{For any positive integers $s, e, m$ there is $0 < \epsilon \leq 1$ depending only on $s, e, m$ with the property that if $K$ is a subgroup of a compact group $G$ and if $G$ has a normal subgroup $T$ of index at most $s$ such that $[K^e, T]$ has order at most $m$, then $K$ is $\epsilon$-central in $G$.}
\medskip

 To see this simply note that if $K$ is as above, then $[G : C_G(g^e)] \leq ms$ for every $g \in K$.

\section{Preliminaries}

In this section we record some results needed in the proofs of the main theorems. The following lemma is \cite[Lemma 3.1]{re}.

\begin{lemma}\label{measure of open subgroup}
Let $G$ be a compact group and let $K$ be a  subgroup of $G$. Then either $\mu(K) = 0$ or $\mu(K) > 0$ and $K$ is open on $G$. Furthermore, in the latter case, $\mu(K) = [G:K]^{-1}$.
\end{lemma}

\begin{lemma}\label{measure estimate on the index}
Let $G$ be a compact group, and let $K$ and $H$ be  subgroups of $G$ with $K \leq H$. Assume further that $\mu(K) \geq \epsilon \mu(H) > 0$ for some positive $\epsilon$. Then $[H:K] \leq \epsilon^{-1}$. 
\end{lemma}
\begin{proof} Since $\mu(K),\mu(H)>0$, the previous lemma implies that both subgroups are of finite index and $\mu(K) = [G:K]^{-1}$ and $\mu(H) = [G:H]^{-1}$. Hence the result.
\end{proof}

For every $x \in G$, the centralizer $C_G(x)$ equals $f_x^{-1}(1)$, where $f_x$ is the continuous function $f_x(y) = [x,y]$, so this subgroup is closed and measurable. 

\begin{lemma}\label{monotone degree}
Let $H$ and $K$ be subgroups of a compact group $G$, with $H \leq K$. Then $$Pr(K, G)\leq Pr(H, G) \leq Pr(H,K).$$ \noindent In particular, $Pr(G, G) \leq Pr(K,G) \leq Pr(K, K)$. 
\end{lemma}
\begin{proof}
Let $\mu$, $\nu$ and $\lambda$ be the normalized Haar measures of $G$, $K$ and $H$, respectively. Given $x \in G$, the map $\alpha: \{hC_H(x) \, | \, h \in H\} \to \{kC_K(x) \, | \, k \in K\}$ taking $hC_H(x)$ to $hC_K(x)$ is injective. We deduce that $[H:C_H(x)] \leq [K:C_K(x)]$ and $\nu(C_K(x)) \leq \lambda(C_H(x)).$  We have \begin{gather*}Pr(H,G) = \int\limits_{G} \lambda(C_H(x)) d\mu(x) \geq \int\limits_{G} \nu(C_K(x)) d\mu(x) = Pr(K, G). \end{gather*} \noindent The other inequality is proved in an analogous way.
\end{proof}

Suppose that $G$ is a compact group and let $N$ be a normal subgroup of $G$. The normalized Haar measure on $G/N$ coincides with the one induced by the normalized Haar measure on $G$. If $A$ is a measurable subset of $G$, we denote by $\chi_A$ the characteristic function of $A$. We say that $x \in G$ is an FC-element if the conjugacy class of $x$ in $G$ is finite.

In the case of finite groups the next lemma was established in \cite{ds}.

\begin{lemma}\label{degree of group and quotient}
Let $G$ be a compact group and let $N$ be a normal subgroup of $G$. For any subgroup $K$ of $G$, we have $$Pr(K,G) \leq Pr(KN/N, G/N)Pr(K \cap N, N).$$
\end{lemma}
\begin{proof}
If $X$ is a compact group, we denote by $\mu_X$ the normalized Haar measure of $X$. We have $$Pr(K,G) = \int\limits_{K} \mu_G(C_G(x))d\mu_K(x).$$ \noindent If $x \in G$ is an FC-element, then $$\mu_G(C_G(x)N) = [C_G(x)N : C_G(x)]\mu_G(C_G(x)) = [N : C_N(x)]\mu_G(C_G(x)) ,$$ \noindent so $\mu_G(C_G(x)N) \mu_N(C_N(x)) = \mu_G(C_G(x)).$ Let $FC(K)$ be the abstract subgroup of $K$ consisting of elements having finite conjugacy class in $G$. Then
\begin{equation*}
\begin{split}
 \int\limits_{K} \mu_G(C_G(x))d\mu_K(x) = &  \int\limits_{FC(K)} \mu_G(C_G(x))d\mu_K(x)\\ = & \int\limits_{FC(K)} \mu_G(C_G(x)N)\mu_{N}(C_N(x))d\mu_K(x) \\ \leq & \int\limits_{K} \mu_G(C_G(x)N)\mu_{N}(C_N(x))d\mu_K(x).
\end{split}
\end{equation*}\noindent We now apply the extended Weil formula \cite[p.88]{reiter} to the last integral and obtain \begin{equation*}
\begin{split} & \hspace{-20pt} Pr(K,G)\leq \\ \int\limits_{\frac{K}{K\cap N}} & \left( \int\limits_{K \cap N}\mu_G(C_G(xk)N)\mu_{N}(C_N(xk))d\mu_{K \cap N}(k)\right) d\mu_{\frac{K}{K\cap N}}(x(K \cap N))   \\ \leq  \int\limits_{\frac{K}{K\cap N}} & \left( \int\limits_{K \cap N}\mu_{\frac{G}{N}}(C_{\frac{G}{N}}(xN))\mu_{N}(C_N(xk))d\mu_{K \cap N}(k)\right) d\mu_{\frac{K}{K\cap N}}(x(K \cap N)) \\ = \int\limits_{\frac{K}{K\cap N}} & \mu_{\frac{G}{N}}(C_{\frac{G}{N}}(xN))\left( \int\limits_{K \cap N}\mu_{N}(C_N(xk))d\mu_{K \cap N}(k)\right) d\mu_{\frac{K}{K\cap N}}(x(K \cap N)). \,\,\, (1)
\end{split}
\end{equation*} \noindent If $x$ is any element of $K$, define the set  \begin{equation*}
\begin{split}A_x & = \{(k, n) \in (K\cap N) \times N \, | \, [xk, n] = 1\}\\ & = \{(k,n) \in (K \cap N) \times N \, | \, xk \in C_G(n) \cap x(K 	\cap N)\}.\end{split}
\end{equation*} \noindent If $C_G(n) \cap x(K\cap N)$ is nonempty, then it equals $tC_{K \cap N}(k)$ for some $t \in x(K \cap N)$. Thus, \begin{equation*}
\begin{split}A_x & = \{(k, n) \in (K\cap N) \times N \, | \, xk \in tC_{K\cap N}(k)\} \\ & = \{(k, n) \in (K\cap N) \times N \, | \, k \in x^{-1}tC_{K\cap N}(k)\}.\end{split}
\end{equation*} \noindent We use the Lebesgue-Fubini Theorem to give an estimate for the expression in parenthesis in (1): \begin{equation*}
\begin{split}
\int\limits_{K \cap N}\mu_{N}(C_N(xk))d\mu_{K \cap N}(k) & = \int\limits_{(K \cap N) \times N} \chi_{A_x}(k,n) d(\mu_{K \cap N} \times \mu_N)(k,n)  \\ & \leq \int\limits_N \mu_{K \cap N}(x^{-1}tC_{K \cap N}(n)) d\mu_N(n) \\ & = \int\limits_N \mu_{K \cap N}(C_{K \cap N}(n)) d\mu_N(n) \\ & = Pr(K \cap N, N).
\end{split}
\end{equation*} \noindent Replacing this back in $(1)$ we have \begin{equation*}
\begin{split} Pr(K,G) \leq 
& \int\limits_{\frac{K}{K\cap N}} \mu_{\frac{G}{N}}(C_{\frac{G}{N}}(xN))\left( \int\limits_{K \cap N}\mu_{N}(C_N(xn))d\mu_{K \cap N}(n)\right) d\mu_{\frac{K}{K\cap N}}(x(K \cap N))\\ \leq  & Pr(K  \cap N, N)\int\limits_{\frac{K}{K\cap N}} \mu_{\frac{G}{N}}(C_{\frac{G}{N}}(xN)) d\mu_{\frac{K}{K\cap N}}(x(K \cap N)).
\end{split}
\end{equation*} Finally, since $K/ K \cap N$ and $KN/N$ are isomorphic, we can apply Corollary 2.5 in \cite{re} with respect to the last integral above to conclude that \begin{equation*}
\begin{split}
\int\limits_{\frac{K}{K\cap N}} \mu_{\frac{G}{N}}(C_{\frac{G}{N}}(xN)) d\mu_{\frac{K}{K\cap N}}(x(K \cap N)) = & \int\limits_{\frac{KN}{N}} \mu_{\frac{G}{N}}(C_{\frac{G}{N}}(xN)) d\mu_{\frac{KN}{N}}(xN) \\ = & Pr(KN/N, G/N).
\end{split}
\end{equation*} \noindent The lemma follows.
\end{proof}

If $A$ and $B$ are normal subgroups of a group $G$ such that $[A:C_A(B)] \leq m$ and $[B:C_B(A)] \leq m$, then $[A,B]$ has $m$-bounded order. This well-known result is due to Baer, cf. \cite[14.5.2]{robinson}. We need a variation of it, which is Lemma 2.1 in \cite{ds}.

\begin{lemma} \label{over abelian subgroup}
Let $m\geq 1$ and let $G$ be a group containing a normal subgroup $A$ and a subgroup $B$ such that $[A:C_A(y)] \leq m$ and $[B:C_B(x)] \leq m$ for all $x \in A, y \in B$. Assume further that $\< B^G \>$ is abelian. Then $[A,B]$ has finite $m$-bounded order. 
\end{lemma}

The next theorem holds in any group and plays a key role in the proof of Theorem \ref{ds adaptation}. It is taken from \cite{as}.

\begin{theorem}\label{acciarri-shumyatsky}
Let $m$ be a positive integer, $G$ a group having a subgroup $K$ such that $|x^G| \leq m$ for each $x \in K$, and let $H = \<K^G\>$. Then the order of the commutator subgroup $[H,H]$ is finite and $m$-bounded.
\end{theorem}

The next lemma is essentially Lemma 2.1 in \cite{eb}. 
\begin{lemma}\label{eberhard}
Let $G$ be a compact group with normalized Haar measure $\mu$ and let $r \geq 1$. Suppose that $X$ is a closed symmetric subset of $G$ containing the identity. If $\mu(X) > \frac{1}{r+1}$, then $\< X \> = X^{3r}$.
\end{lemma}
\begin{proof}
Suppose $x_i \in X^{3i+1} \setminus X^{3i}$ for $i = 0, \dots, r$. Then for each $i$, as long as $X^{3i+1} \setminus X^{3i}$ is nonempty, we have $$x_iX \subseteq X^{3i+2} \setminus X^{3i-1}.$$ \noindent So, assume that the sets $X^{3i+1} \setminus X^{3i}$ are nonempty for $i=0,\dots, r$. Then $x_0X, \dots, x_rX$ are disjoint subsets of $G$, each of measure $\mu(X)$, and $$\mu\left(\bigcup_{i=0}^r x_iX\right) = (r+1)\mu(X) > 1.$$ \noindent Therefore $X^{3i+1} =  X^{3i}$ for some $i \leq r$. In particular, $X^{3r} = \<X\>.$ 
\end{proof}

In a compact group $G$ the set  $\{x \in G \, | \, |x^G| \leq n\}$ is closed and thus measurable. 

\begin{lemma}\label{fc-elements}
Let $G$ be a compact group and $n$ a positive integer. The set $X = \{x \in G \, | \, |x^G| \leq n\}$ is a closed subset of $G$.
\end{lemma}
\begin{proof}
It is sufficient to show that if $a \in G \setminus X$, then $a$ is contained in an open subset $U$ which has empty intersection with $X$. Since $a \notin X$, we can choose $n+1$ elements $x_1,...,x_{n+1}$ in such a way that the conjugates $a^{x_i}$ are distinct for $i=1,\dots,n+1$. Set $$U=\{u\in G;\ [u,x_ix_j^{-1}]\neq1 \text{ for } 1\leq i,j\leq n+1\}.$$

Observe that $a\in U$ and every element in $U$ has at least $n+1$ conjugates, whence $U\cap X=\emptyset$. Further, since the commutator map is continuous, $U$ is open. The proof is complete.
\end{proof}

\begin{remark}\label{remark}\textnormal {If $S$ is any subset of $G$, we denote by $x^S$ the set of $S$-conjugates of $x$ in $G$. The previous argument also proves that the set $\{x \in G \, | \, |x^S| \leq n\}$ is closed in $G$ for any (not necessarily closed) subset $S$ of $G$.}
\end{remark}

\section{Proof of Proposition \ref{ds adaptation}}

Now we are able to prove Proposition \ref{ds adaptation}, which we restate here for the reader's convenience. 
\medskip

\textit{
Let $\epsilon > 0$ and let $G$ be a compact group having a subgroup $K$ such that $Pr(K,G) \geq \epsilon$. Then there is a normal subgroup $T \leq G$ and a subgroup $B \leq K$ such that the indices $[G:T]$ and $[K:B]$ and the order of $[T,B]$ are $\epsilon$-bounded. 
}

 \begin{proof}Let $\mu$ and $\nu$ be the normalized Haar measures of $G$ and $K$, respectively. Set $$X = \{x \in K \, | \, | x^G| \leq 2/\epsilon\}.$$

\noindent Note that $X$ is measurable, by Lemma \ref{fc-elements}: it is the intersection of $K$ and the closed set $\{x \in G \, | \, |x^G| \leq 2/\epsilon\}$. We have $$K \setminus X = \{x \in K \, | \, |x^G| > 2/\epsilon\}.$$ \noindent  Since $\mu(C_G(x)) < \epsilon/2$ for all $x \in K\setminus X$, it follows that
 \vspace{-5pt}\begin{equation*}
    \begin{split}
      \epsilon \leq  Pr(K,G) = & \int\limits_K \mu(C_G(x)) d\nu(x) \\
                = & \int\limits_X \mu(C_G(x)) d\nu(x) + \int\limits_{K\setminus X} \mu(C_G(x)) d\nu(x) \\
                \leq & \int\limits_X d\nu(x) + \int\limits_{K \setminus X} \frac{\epsilon}{2} d\nu(x) \\
                = & \nu(X) + \frac{\epsilon}{2} (1 - \nu(X)) \leq \nu(X) + \frac{\epsilon}{2}.
    \end{split}
\end{equation*} \noindent This implies that $\epsilon/2 \leq \nu(X).$ Let $B$ be the subgroup generated by $X$. Then, by Lemma \ref{eberhard}, every element of $B$ is a product of at most $6/\epsilon$ elements of $X$. Clearly, $\nu(B) \geq \nu(X) \geq \epsilon/2$, so the index of $B$ in $K$ is at most $2/\epsilon$, by Lemma \ref{measure estimate on the index}. Furthermore, $|b^G| \leq (2/\epsilon)^{6/\epsilon}$ for every $b \in B$. 

Let $L = \< B^G\>$. Theorem \ref{acciarri-shumyatsky} tells us that the commutator subgroup $[L,L]$ has finite $\epsilon$-bounded order. Let us use the bar notation for images of subgroups of $G$ in $G/[L,L]$. By Lemma \ref{degree of group and quotient}, $Pr(\overline{K}, \overline{G}) \geq Pr(K,G) \geq \epsilon.$  Moreover, $[\overline{K}:\overline{B}] \leq [K:B] \leq \epsilon/2$ and $|\overline{b}^{\overline{G}}| \leq |b^G| \leq (2/\epsilon)^{6/\epsilon}$ for any $b \in B$. Thus we can pass to the quotient over $[L,L]$ and assume that $L$ is abelian. 

Now we set $$Y = \{y \in G \, | \, |y^K| \leq 2/\epsilon \}.$$ \noindent Observe that $Y$ is closed, by Remark \ref{remark}. 
Arguing as before, since $\nu(C_K(y)) < \epsilon/2$ for all $y \in G \setminus Y$, we have \begin{equation*}
    \begin{split}
        \epsilon \leq Pr(K,G) = &  \int\limits_G \nu(C_K(y)) d\mu(y) \\
                = & \int\limits_Y \nu(C_K(y)) d\mu(y) + \int\limits_{G\setminus Y} \nu(C_K(y)) d\mu(y) \\
                \leq & \int\limits_Y d\mu(y) + \int\limits_{G \setminus Y} \epsilon/2 d\mu(y) \\
                = & \mu(Y) + \epsilon/2 (1 - \mu(Y)) \leq \mu(Y) + \epsilon/2.
    \end{split}
\end{equation*}

\noindent Therefore, $\mu(Y) \geq \epsilon/2$. Let $E$ be the subgroup generated by $Y$. Lemma \ref{eberhard} ensures that every element of $E$ is a product of at most $6/\epsilon$ elements of $Y$. Also, we have $\mu(E) \geq \mu(Y) \geq \epsilon/2$, so the index of $E$ in $G$ is at most $2/\epsilon$, by Lemma \ref{measure estimate on the index}.  Since $|y^K| \leq 2/\epsilon$ for every $y \in Y$, it follows that $|g^K| \leq (2/\epsilon)^{6/\epsilon}$ for every $g \in E$. Let $T$ be the maximal normal subgroup of $G$ contained in $E$. Then the index $[G:T]$ is $\epsilon$-bounded. Moreover, $|b^G| \leq (2/\epsilon)^{6/\epsilon}$ for every $b \in B$ and $|g^K| \leq (2/\epsilon)^{6/\epsilon}$ for every $g \in T$. As $L$ is abelian, we can apply Lemma \ref{over abelian subgroup} and deduce that $[T,B]$ has finite $\epsilon$-bounded order. The proposition follows. 
\end{proof}

\section{About $\epsilon$-central subgroups}

Let $G$ be a topological group generated by a symmetric set $X$. If it is possible to write $g\in G$ as a product of finitely many elements from $X$, we denote by $w(g)$ the shortest length of such an expression. If $g$ cannot be written as a product of finitely many elements of $X$, we simply say that $w(g)$ is infinite. The next result is Lemma 2.1 in \cite{dierings-s}.
\begin{lemma}\label{glaucia}
Let $G$ be a group generated by a symmetric set $X$ and let $D$ be a subgroup of index $m$ in $G$. Then every coset $Db$ contains an element such that $w(g) \leq m-1$.
\end{lemma}
We remark that Lemma \ref{glaucia} holds for topological groups and their closed (open) subgroups of finite index. Indeed, for an integer $r\geq 0$ let $D_r$ be the union of the cosets of $D$ containing some element $g$ with $w(g)\leq r$. Then $D_r \subseteq D_{r+1}$ and $D_rX\subseteq D_{r+1}$ for all $r$. Let $R$ be the minimal integer such that $D_{R+1} = D_R$. Then $D_R$ is a closed set containing the group generated by $X$, so $D_R =G$. Since $D$ has $m$ cosets and $D_0 =D$, $R<m$.

We now proceed to the proof of Proposition \ref{proposition}, which we restate here for the reader's convenience. 
\medskip

\textit{Let $G$ be a compact group and let $l, n$ be positive integers. Suppose that there is a subgroup $K$ of $G$ such that $[G:C_G(g^l)] \leq n$ for every $g \in K$. Then there exists a positive integer $e$, depending only on $l$ and $n$, and a normal subgroup $T$ of $G$, such that the index $[G:T]$ and the order of $[K^e,T]$ both are finite and $(l,n)$-bounded.}
\medskip

Let $G$ be a compact group satisfying the hypothesis of Proposition \ref{proposition}. Let $X$ be the union of the conjugacy classes of $G$ containing an $l$th power of an element of $K$ and let $H$ be the subgroup generated by $X$.  Define $m$ as the maximum of the indices of $C_H(x)$ in $H$, where $x \in X$. Obviously, $m \leq n$. 

\begin{lemma}\label{boundedd}
For any $x \in X$ the order of the subgroup $[H,x]$ is $m$-bounded. 
\end{lemma}
\begin{proof}
Since the index of $C_H(x)$ in $H$ is at most $m$, Lemma \ref{glaucia} guarantees that there are elements $y_1, \dots, y_m$ in $H$ such that each $y_i$ is a product of at most $m-1$ elements of $X$ and the subgroup $[H,x]$ is generated by the commutators $[y_i, x]$, for $i = 1, \dots, m$. For any such $i$ write $y_i = y_{i1}\dots y_{i(m-1)}$, where $y_{ij}$ belongs to $X$. Using the standard commutator identities, we can rewrite $[y_i, x]$ as a product of conjugates in $H$ of the commutators $[y_{ij}, x]$. Let $\{h_1, \dots, h_s\}$ be the set of conjugates in $H$ of all elements from the set $\{x, y_{ij} \, |\, 1 \leq i,j\leq m-1\}$. Note that the number $s$ here is $m$-bounded. This follows from the fact that $C_H(x)$ has index at most $m$ in $H$ for every $x \in X$. Let $D$ be the subgroup of $H$ generated by $h_1, \dots, h_s$. Since $[H,x]$ is contained in the commutator subgroup $D'$, it suffices to show that $D'$ has finite $m$-bounded order. Observe that the center $Z(D)$ has index at most $m^s$ in $D$, since the index of $C_H(h_i)$ is at most $m$ for every $h_i$. Thus, by Schur's theorem \cite[10.1.4]{robinson}, we conclude that $D'$ has finite $m$-bounded order. 
\end{proof}

We now argue by induction on $m$. If $[H:C_H(g^{l^2})] \leq m-1$ for all $g \in K$, then by induction the result holds. We therefore assume that there is $d \in K$ such that $[H:C_H(d^{l^2})] = m$. Of course, necessarily, $[H:C_H(d^l)] = m$. Set $a = d^l$ and choose $b_1, \dots, b_m$ in $H$ such that $a^H = \{a^{b_i}\, | \, i = 1, \dots, m\}$ and $w(b_i) \leq m-1$ (the existence of the elements $b_i$ is guaranteed by Lemma \ref{glaucia}). Since $C_H(a) = C_H(a^l)$, it follows that $(a^l)^H = \{(a^l)^{b_i} \, | \, i=1,\dots, m\}$. Set $U = C_G(\<b_1,\dots,b_m\>)$. Note that the index of $U$ in $G$ is $n$-bounded. Indeed, since $w(b_i) \leq m-1$ we can write $b_i = b_{i1}\dots b_{i(m-1)}$, where $b_{ij} \in X$ and $i = 1, \dots, m$. By the hypothesis, the index of $C_G(b_{ij})$ in $G$ is at most $n$ for any such element $b_{ij}$. Thus, $[G:U] \leq n^{(m-1)m}$.

\begin{lemma}\label{bounded}
Suppose that $u \in U$ and $ua \in X$. Then $[H,u] \leq [H,a]$.
\end{lemma}
\begin{proof}
For each $i = 1, \dots, m$ we have $(ua)^{b_i} = ua^{b_i}$, since $u$ belongs to $U$. By hypothesis, $ua \in X$. Hence, taking into account the assumption on the cardinality of the conjugacy class of $ua$ in $H$, we deduce that $(ua)^H$ consists exactly on the elements $(ua)^{b_i}$, for $i = 1, \dots, m$. Therefore, given an arbitrary element $h \in H$, there exists $b \in \{b_1, \dots, b_m\}$ such that $(ua)^h = (ua)^b$ and so $u^ha^h = ua^b$. It follows that $[u,h] = a^ba^{-h} \in [H,a]$, and the lemma holds. 
\end{proof}

Let $R$ be the normal closure in $G$ of the subgroup $[H,a]$, that is, $R = [H,a^{b_1}]\dots[H,a^{b_n}]$, where $a^{b_i}$ are all the conjugates of $a$ in $G$ (if $|a^G| \leq n-1$, then not all the $a^{b_1}, \dots, a^{b_n}$ are pairwise distinct). By Lemma \ref{boundedd}, each of the subgroups $[H, a^{b_i}]$ has $n$-bounded order. Thus, the order of $R$ is $n$-bounded as well.

Let $Y_1 = Xa^{-l} \cap U$ and $Y_2 = Xa^{-1} \cap U$. Note that for any $y \in Y_1$, the product $ya^l$ belongs to $X$. So, by Lemma \ref{bounded} applied with $a^l$ in place of $a$ and $y$ in place of $u$, the subgroup $[H,y]$ is contained in $[H,a^l]$, which is contained in $R$. Similarly, for any $y\in Y_2$, we have $[H,y] \leq R$. Set $Y = Y_1\cup Y_2$. Thus, $[H,Y]\leq R$. 

Let $U_0$ be the maximal normal subgroup of $G$ contained in $U$. Observe that the index of $U_0$ in $G$ is $n$-bounded. Observe further that for any $u \in U_0$ the commutators $[u,a^{-l}]$ and $[u,a^{-1}]$ lie in $Y$. Since $[U_0, a^{-1}] = [U_0, a]$, we deduce that $$[H,[U_0,a]] \leq R.$$ \noindent Let $K_0 = K \cap U_0$. We remark that  $(ua)^l(a^l)^{-1} \in Y$ whenever $u \in K_0$. We pass to the quotient $\overline{G} = G/R$ and use the bar notation to denote images in $\overline{G}$. We know that $\overline{Y}$ is central in $\overline{H}$. We also deduce that $[\overline{U}_0, \overline{a}] \leq Z(\overline{H})$. 

Since $(ua)^l(a^l)^{-1} \in Y$ whenever $u \in K_0$ and since $\overline{Y} \leq Z(\overline{H})$, it follows that in the quotient $\overline{G}/Z(\overline{H})$ the element $\overline{a}$ commutes with $\overline{U}_0$ and $(\overline{u}\overline{a})^l(\overline{a})^{-l} = 1$ for every $\overline{u} \in \overline{K}_0$. This implies that $\overline{K}_0$ has exponent dividing $l$ modulo $Z(\overline{H})$. It follows that $\overline{K}_0^l$ is abelian and every element of $\overline{K_0}^{l^2}$ is again an $l$th power of an element in $\overline{K}_0$. We therefore deduce that $$Pr(\overline{K}_0^{l^2}, \overline{G}) \geq \frac{1}{n}.$$

By Proposition \ref{ds adaptation} there is a normal subgroup $\overline{T}$ in $\overline{G}$ and a subgroup $\overline{V}$ in $\overline{K}_0^{l^2}$ such that the indices $[\overline{G}:\overline{T}]$ and $[\overline{K}_0^{l^2}:\overline{V}]$ and the order of $[\overline{T}, \overline{V}]$ are $(l,n)$-bounded. Let $T$ be the inverse image of $\overline{T}$ in $G$ and $V$ the inverse image of $\overline{V}$ in $K_0^{l^2}$. Bearing in mind that the order of $R$ is $n$-bounded, we conclude that the indices $[G:T]$ and $[K_0^{l^2}:V]$ are $(n,l)$-bounded, as also is the order of $[T,V]$. As the index of $V$ in $K_0^{l^2}$ is bounded, there is a positive $(n,l)$-bounded integer $e$ such that $K^e \leq V$. This completes the proof of the proposition. \hfill $\qed$

\begin{corollary} Let $G$ be a compact group and let $G_0$ be the connected component of identity in $G$. Then, the following are equivalent 
\begin{itemize}
\item[(i)] The probability $Pr(G_0, G)$ is positive;
\item[(ii)] The centralizer $C_G(G_0)$ is open in $G$;
\item[(iii)] The connected component $G_0$ is $\epsilon$-central in $G$ for some $\epsilon >0$.
\end{itemize}
\end{corollary}
\begin{proof}
Suppose first that $Pr(G_0, G) > 0$. Then there are subgroups of finite index $T$ of $G$ and $B$ of $G_0$ such that $[T, B]$ is finite. Since $G_0$ is divisible \cite{my}, the only finite index subgroup of $G_0$ is $G_0$ itself. Therefore, the set of commutators $\{[x, y] \, | \, x \in G_0, y \in T\}$ is connected and finite, hence is trivial and $T \leq C_G(G_0)$. 

Now, assume that $C_G(G_0)$ is open in $G$ and let $\mu$ and $\mu_0$ be the normalized Haar measures of $G$ and $G_0$, respectively. For any $x \in G_0$, the inclusion $C_G(G_0) \leq C_G(x)$ holds, thus $\mu(C_G(G_0)) \leq \mu(C_G(x))$. We have $$Pr(\<x\>, G) = \int\limits_{G_0} \mu(C_G(x)) d\mu_0(x) \geq \mu(C_G(G_0)) > 0.$$ \noindent We conclude that $G_0$ is $\epsilon$-central, where $\epsilon = [G:C_G(G_0)]^{-1}$, and so (ii) implies (iii).

Now we assume the validity of (iii) and prove that (i) holds. By Theorem \ref{pavel}, there are a finite index subgroup $T$ of $G$ and a natural number $e$ such that $[G_0^e, T]$ is finite. Since $G_0$ is divisible, $G_0 = G_0^e$ and so $T$ centralizes $G_0$ and $T \leq C_G(x)$ for any $x \in G_0$. Writing $\mu$ and $\mu_0$ for the normalized Haar measures of $G$ and $G_0$, respectively, we have $$Pr(G_0, G) = \int\limits_{G_0} \mu(C_G(x)) d\mu_0(x) \geq   \mu(T) > 0.$$ 
\end{proof}

\section{Corollaries for Finite Groups} 

In this section we collect some easy corollaries of Theorem \ref{pavel} for finite groups. Roughly, we show that many well-known results on the exponent of a finite group admit a probabilistic interpretation in the spirit of Theorem \ref{pavel}.

The restricted Burnside problem was whether the order of an $r$-generator finite group of exponent $e$ is bounded in terms of $r$ and $e$ alone. This was famously solved in the affirmative by Zelmanov \cite{ze1, ze2}. Theorem \ref{pavel} enables us to obtain the following extension of Zelmanov's theorem.

\begin{theorem}
Let $G$ be a finite $\epsilon$-central $r$-generator group. Then $G$ has a normal subgroup $N$ such that both the index $[G : N]$ and the order of the commutator subgroup $[N,N]$  are $(r, \epsilon)$-bounded.
\end{theorem}
\begin{proof}
By Theorem \ref{pavel} there is an $\epsilon$-bounded number $e$ and a normal subgroup $T \leq G$ such that the index $[G : T]$ and the order of $[G^e, T]$ are $\epsilon$-bounded. The solution of the Restricted Burnside problem implies that the index $[G:G^e]$ is $(e,r)$-bounded. Therefore the subgroup $N = G^e \cap T$ has the required properties.
\end{proof}

An important part of the eventual solution of the restricted Burnside problem was developed by Hall and Higman in their paper \cite{hh}. They proved that if $p$ is a prime and $G$ is a finite group with Sylow $p$-subgroups of exponent $p^s$, then $G$ has a normal series of $s$-bounded length all of whose factors are $p$-groups, or $p'$-groups, or direct products of nonabelian simple groups of order divisible by $p$. 

\begin{theorem}\label{probabilistic hall-higman}
Let p be a prime and G a finite group with $\epsilon$-central Sylow $p$-subgroups. Then $G$ has a normal series of $\epsilon$-bounded length
all of whose factors are $p$-groups, or $p'$-groups, or direct products of nonabelian simple groups of order divisible by $p$.
\end{theorem}
\begin{proof}
 Let $P$ be a Sylow $p$-subgroup of $G$. By Theorem \ref{pavel} there is an $\epsilon$-bounded number $e$ and a normal subgroup $T \leq  G$ such that the index $[G : T]$ and the order of $[P^e, T]$ are $\epsilon$-bounded. Set $P_0 = P^e$. It is sufficient to show that $T$ has a normal series with the required properties. Note that $[P_0, T]$ is normal in $T$. Let $Z$ be the inverse image in $T$ of the center of $T/[P_0, T]$. Consider the normal series $$1\leq [P_0, T] \leq Z \leq T.$$ \noindent  Here $[P_0, T]$ has $\epsilon$-bounded order, $Z/[P_0, T]$ is abelian, and $T/Z$ has Sylow $p$-subgroups of exponent dividing $e$. According to the Hall-Higman theory $T/Z$ has a normal series of $\epsilon$-bounded length with the required properties. Thus, the result follows.
\end{proof}

Given a group-word $w$, we write $w(G)$ for the corresponding verbal subgroup of a group $G$, that is, the subgroup generated by the values of $w$ in $G$. The word is said to be a law in $G$ if $w(G) = 1$. In view of Theorem \ref{pavel} it is easy to see that if $w(G)$ is $\epsilon$-central in $G$, then a law of $(\epsilon,w)$-bounded length holds in the group $G$. This simple observation provides a tool for obtaining extensions of results about finite groups satisfying certain laws. We will illustrate this with a theorem bounding the nonsoluble length of a finite group. The concept of nonsoluble length $\lambda(G)$ of a finite group $G$ was introduced in \cite{ks}. This is the minimum number of nonsoluble factors in a normal series of $G$ in which every factor either is soluble or is a direct product of non-abelian simple groups. (In particular, the group is soluble if and only if its nonsoluble length is 0.)

It was shown in \cite{fumagalli} that if a word $w$ is a law in the Sylow 2-subgroup of a finite group $G$, then $\lambda(G)$ is bounded in terms of the length of the word $w$ only. This can be extended as follows. 
\begin{theorem}
Let $w$ be a group-word and $P$ a Sylow $2$-subgroup of a finite group $G$ such that $w(P)$ is $\epsilon$-central in $P$. Then $\lambda(G)$ is
$(\epsilon, w)$-bounded.
\end{theorem}
\begin{proof}
As explained above, this is straightforward combining the result in \cite{fumagalli} and Theorem \ref{pavel}.
\end{proof}

For a group of automorphisms $A$ of a group $G$ we write $C_G(A)$ for the centralizer of $A$ in $G$. It is well-known that if a finite group $G$ admits a coprime group of automorphisms $A$, then $C_{G/N}(A) = NC_G(A)/N$ for any $A$-invariant normal subgroup $N$ of $ G$ (see for example \cite[Theorem 6.2.2 (iv)]{gor}). Here the group $A$ is a coprime group of automorphisms if $(|G|, |A|) = 1$. The symbol $A^{\#}$ stands for the set of nontrivial elements of the group $A$. 
The main result of \cite{khushu} states that if a finite group $G$ admits an elementary abelian coprime group of automorphisms $A$ of order $p^2$ such that $C_G(\phi)$ has exponent dividing $\epsilon$ for each $\phi \in A^{\#}$, then the exponent of $G$ is $(\epsilon,p)$-bounded. We can now extend this in the following way.

\begin{theorem}\label{probabilistic autos and exp}
 Let $\epsilon > 0$, and let $G$ be a finite group admitting an elementary abelian coprime group of automorphisms $A$ of order $p^2$ such that $C_G(\phi)$ is $\epsilon$-central in $G$ for each $\phi \in A^{\#}$. Then there is a $(p, \epsilon)$-bounded number $e$ and an $A$-invariant normal subgroup $T$ such that the index $[G : T]$ and the order of $[G^e, T]$ are $(p, \epsilon)$-bounded.
\end{theorem}
\begin{proof} Let $A_1, \dots, A_{p+1}$  be the subgroups of order $p$ in $A$ and set $G_i = C_G(A_i)$ for $i = 1, \dots, p + 1$. According to Theorem \ref{pavel} there is an $\epsilon$-bounded number $d$ and, for $i = 1, \dots, p + 1$, $A$-invariant normal subgroups $T_i \leq G$ such that the index $[G : T_i]$ and the order of $[G_i^d, T_i]$ are $\epsilon$-bounded. Set $T = \bigcap T_i$ and observe that $T$ is $A$-invariant and the index of $T$ in $G$ is $(p, \epsilon)$-bounded. Let $N_i = [G_i^d, T]$ and $N_0 = \prod N_i$. Note that the subgroup $N_0$ is normal in $T$ and has $(p, \epsilon)$-bounded order. Let $N = \<N_0^G\>$ be the normal closure of $N_0$ in $G$. Since the index of $T$ in $G$ is $(p, \epsilon)$-bounded, it follows that the order of $N$ is $(p, \epsilon)$-bounded as well. We also observe that $N$ is $A$-invariant since the subgroups $N_i$ are. Let $C$ be the centralizer of $T$ modulo $N$, that is, $C = \{ x \in G \,| \,  [T, x] \leq N \}$. Clearly, the subgroup $C$ is $A$-invariant. Moreover $G_i^d\leq C$ for $i = 1, \dots,  p + 1$. Hence, $C_{G/C}(A_i)$ has exponent dividing $d$ for each $i = 1, \dots, p + 1$. Now the main result of \cite{khushu} says that the exponent of $G/C$ is $(d, p)$-bounded. Therefore there exists a $(p, \epsilon)$-bounded number $e$ such that $G^e \leq C$, that is, $[G^e, T] \leq N$. This completes the proof. 
\end{proof}
\begin{corollary}
Under the hypotheses of Theorem \ref{probabilistic autos and exp} there exists a number $\epsilon_0 > 0$ depending only on $\epsilon$ and $p$ such that $G$ is $\epsilon_0$-central.
\end{corollary}
\begin{proof}
Assume the hypotheses of Theorem \ref{probabilistic autos and exp}. The theorem tells us that there is a $(p, \epsilon)$-bounded number $e$ and a normal subgroup $T$ such that the index $[G : T]$ and the order of $[G^e, T]$ are $(p, \epsilon)$-bounded. As explained in the introduction, Theorem \ref{pavel} admits a converse. Hence the result. 
\end{proof}
In the spirit of the work \cite{gs} we record the following theorem.
\begin{theorem}
 Let $\epsilon > 0$, and let $G$ be a finite group admitting an elementary abelian coprime group of automorphisms $A$ of order $p^3$ such that the commutator subgroup of $C_G(\phi)$ is $\epsilon$-central in $G$ for each $\phi \in A^{\#}$. Then there is a $(p, \epsilon)$-bounded number $e$ and an $A$-invariant normal subgroup $T$ such that the index $[G : T]$ and the order of $[[G,G]^e, T]$ are $(p, \epsilon)$-bounded. 
\end{theorem}
\begin{proof} Let $A_1, \dots, A_s$ be the subgroups of order $p$ of $A$ and let $D_i$ denote the commutator subgroup of $C_G(A_i)$ for $ i = 1, \dots,  s$. According to Theorem \ref{pavel} there is an $\epsilon$-bounded number $d$ and, for $i = 1, \dots, s$, $A$-invariant normal subgroups $T_i \leq G$ such that the index $[G : T_i]$ and the order of $[D_i^d, T_i]$ are $\epsilon$-bounded. Set $T = \bigcap T_i$ and observe that $T$ is $A$-invariant and the index of $T$ in $G$ is $(p, \epsilon)$-bounded. Let $N_i = [D_i^d, T]$ and $N_0 = \prod N_i$. Note that $N_0$ is normal in $T$ and has $(p, \epsilon)$-bounded order. Let $N = \<N_0^G \>$ be the normal closure of $N_0$ in $G$. Since the index of $T$ in $G$ is $(p, \epsilon)$-bounded, it follows that the order of $N$ is $(p, \epsilon)$-bounded. We also observe that $N$ is $A$-invariant since the subgroups $N_i$ are. Let $C$ be the centralizer of $T$ modulo $N$, that is, $C = \{x \in G\, | \, [T, x] \leq N\}$. Clearly, the subgroup $C$ is $A$-invariant. Moreover $D_i^d \leq C$ for $i = 1, \dots, s$. Hence, $C_{G/C}(A_i)$ has commutator subgroup of exponent dividing $d$ for each $i = 1, \dots, s$. Now the main result of \cite{gs} says that the exponent of the commutator subgroup of $G/C$ is $(d, p)$-bounded. Therefore there exists a $(p, \epsilon)$-bounded number $e$ such that $[G,G]^e\leq C$, that is, $[[G,G]^e, T] \leq N$. This completes the proof.
\end{proof}

\end{document}